\documentclass{amsart}
\usepackage{amsmath}
\usepackage{amssymb}
\usepackage{latexsym}
\usepackage{amsthm}
\usepackage{enumerate}
\usepackage{ulem}
\usepackage[dvips]{graphicx}
\theoremstyle{definition}
\newtheorem{definition}{Definition}[section]
\theoremstyle{plain}
\newtheorem{lemma}[definition]{Lemma}
\newtheorem{theorem}[definition]{Theorem}
\newtheorem{proposition}[definition]{Proposition}
\newtheorem{corollary}[definition]{Corollary}
\theoremstyle{remark}
\newtheorem{remark}[definition]{Remark}

\makeatletter
\@namedef{subjclassname@2020}{
  \textup{2020} Mathematics Subject Classification}
\makeatother

\begin{document}

\title[Expansion of semi-bounded o-minimal structure]{Expansion of a semi-bounded o-minimal structure by a geometric progression}
\author[M. Fujita]{Masato Fujita}
\address{Department of Liberal Arts,
Japan Coast Guard Academy,
5-1 Wakaba-cho, Kure, Hiroshima 737-8512, Japan}
\email{fujita.masato.p34@kyoto-u.jp}

\begin{abstract}
We demonstrate that an expansion of a semi-bounded o-minimal expansion of the ordered group of reals by an increasing geometric progression is locally o-minimal.
\end{abstract}



\maketitle

\section{Introduction}\label{sec:intro}
\sout{The author proposed new locally o-minimal structures whose universe is the set of reals. }
{\bf{The author realized that Theorem \ref{thm:example} follows from \cite[Theorem 1]{MT} from the personal communication with C. Miller.}}

Locally o-minimal structures are extensively studied in \cite{TV, Fornasiero, KTTT,Fuji1,Fuji2,Fuji3,Fuji4} and these study demonstrates that sets definable in definably complete locally o-minimal structures enjoy tame topological properties.
Note that locally o-minimal expansions of the set of reals are almost o-minimal.
An expansion of dense linear order without endpoints $\mathcal R=(R,<,\ldots)$ are called \textit{almost o-minimal} \cite{Fuji} if every bounded definable subset of $R$ is the union of finitely many points and finitely many open intervals.
The author introduced the notion of almost o-minimal structures because their topological properties are more tame than other locally o-minimal structures. 
For instance, an almost o-minimal expansion of an ordered group admits uniform local definable cell decomposition \cite[Theorem 1.7]{Fuji}, but a locally o-minimal structure does not necessarily admit local definable cell decomposition by \cite[Corollary 4.1]{Fuji1}.

We introduce known examples of locally o-minimal expansion of the ordered group of reals.
We define a quasi-periodic locally o-minimal expansion of the ordered group of reals later.
In such a structure, there exist a positive $r \in \mathbb R$ and an o-minimal structure $\mathcal R'$ such that any definable set is a finite union of sets of the form $$\bigcup_{a \in Z}a+X,$$
where $Z$ is a definable subset of $(r\mathbb Z)^n$ and $X$ is a subset of $[0,r)^n$ definable in $\mathcal R'$.
Many examples of locally o-minimal structures fall into this category including the structure $(\mathbb R,<,+,\sin x)$ given as an example of locally o-minimal structures in \cite{TV}.

There are locally o-minimal structures not fallen into this category.
Proposition 26 and Remark 27 of \cite{KTTT} demonstrate that $\mathcal R=(\mathbb R,+,<,\{e^n\;|\;n < \omega\})$ is locally o-minimal and $(\mathcal R,\mathbb Z)$ is not, where $e$ is the base of logarithm.
Another example was provided by Friedman and Miller.
Consider a strictly increasing sequence $\{\phi_k\;|\;k \in \mathbb N\}$ of positive real numbers such that $\lim_{k \to \infty}\phi_k/\phi_{k+1}=0$.
The expansion of  a semi-bounded o-minimal structure  expanded by the sequence $\{\phi_k\}$ is locally o-minimal \cite[Theorem 3]{FM}.

In this paper, we give one another example.
Let $\mathcal R=(\mathbb R,<,+,0,\ldots)$ be a semi-bounded expansion of the ordered group of reals.
We show that the expansion $\mathcal R_E$ of $\mathcal R$ by the geometric progression $E=\{\rho^n\;|\; n \geq 0\}$ with $\rho>1$ is a non-quasi-periodic locally o-minimal structure and any bounded set definable in $\mathcal R_E$ is definable in $\mathcal R$.

This paper is organized as follows:
We demonstrate that $\mathcal R_E$ is locally o-minimal and any bounded set definable in $\mathcal R_E$ is definable in $\mathcal R$ in Section \ref{sec:almost_ominimal}.
We define quasi-periodic o-minimal structure and show that $\mathcal R_E$ is not quasi-periodic in Section \ref{sec:periodic}.
We introduced an example \cite[Theorem 3]{FM} given by Friedman and Miller.
However, its original proof has a gap.
We fix it in Appendix \ref{sec:appendix}.

\section{Local o-minimality}\label{sec:almost_ominimal}
We demonstrate that the expansion $\mathcal R_E$ is locally o-minimal in this section.
We first recall the definition of semi-bounded o-minimal structures.
\begin{definition}[\cite{Edmundo}]
	An o-minimal expansion $\mathcal R=(R,<,+,0,\ldots)$ of an ordered group is \textit{semi-bounded} if every definable set is already definable in the reduct $(R,0,1,+,<,(B_i)_{i \in I}, (\lambda)_{\lambda \in \Lambda})$, where $(B_i)_{i \in I}$ is the collection of all bounded definable sets and $\Lambda$ is the collection of all definable endomorphisms on $R$.
\end{definition}

%%%%%%%%%%%%%

\begin{lemma}\label{lem:example_basic_pre}
	Let $\rho$ be a real number with $\rho>1$.
	Set $E=\{\rho^n\;|\; n \in \mathbb Z,\ n \geq 0\}$.
	For a positive integer $m$, a sequence of real numbers $(a_1, \ldots, a_m)$ of length $m$ and a positive real number $R$, we put
	$$\mathfrak I((a_1,\ldots, a_m),R)= v(E^m) \cap (-R,R),$$
	where $v$ is the linear map given by $v(x_1,\ldots, x_m) = \sum_{i=1}^m a_ix_i$.
	Then the set $\mathfrak I((a_1,\ldots, a_m),R)$ is a finite set and there exists $r>0$ such that $v(E^m) \cap (-r,r) \subseteq \{0\}$.
\end{lemma}
\begin{proof}
	We prove the lemma by induction on $m$.
	The lemma is obvious when $m=1$.
	
	We next consider the case in which $m>1$.
	The first task is to show that $\mathfrak I((a_1,\ldots, a_m),R)$ is a finite set.
	We set
	\begin{align*}
	  & \mathfrak J_i((a_1,\ldots, a_m),R) = \left\{\left.\sum_{j=1}^n a_jx_j\;\right|\; x_j \in E,\ x_j \geq x_i \ \ (1 \leq j \leq m)\right\} \cap (-R,R)%\\
	\end{align*}
	for each $1 \leq i \leq m$.
We have only to demonstrate that $\mathfrak J_i((a_1,\ldots, a_m),R)$ is a finite set for each $1 \leq i \leq m$ because $\mathfrak I((a_1,\ldots, a_m),R) = \bigcup_{i=1}^m \mathfrak J_i((a_1,\ldots, a_m),R)$.
We only consider the case in which $i=1$.
The proof is similar in the other cases.

We consider the sets 
	\begin{align*}
	  & \mathfrak K(k,(a_1,\ldots, a_m)) = \left\{\left.\sum_{j=1}^n a_jx_j\;\right|\; x_j \in E,\ x_j \geq x_1 =\rho^k\ \ (1 \leq j \leq m)\right\} \text{ and }\\
	  & \mathfrak L(k,(a_1,\ldots, a_m),R) = \mathfrak K(k,(a_1,\ldots, a_m))  \cap (-R,R)
\end{align*}
for $k \geq 0$.
Since $|\sum_{i=2}^m a_jx_j| \leq |a_1|+|a_1+\sum_{j=2}^m a_jx_j| $, we have $$\mathfrak L(0,(a_1,\ldots, a_m),R) \subseteq  a_1+\mathfrak I((a_2,\ldots, a_m),R+|a_1|).$$
The set $\mathfrak I((a_2,\ldots, a_m),R+|a_1|)$ is a finite set by the induction hypothesis.
It implies that the set $\mathfrak L(0,(a_1,\ldots, a_m),R)$ is also finite.
The equality 
\begin{equation*}
\mathfrak K(k,(a_1,\ldots, a_m)) =\rho^k \mathfrak K(0,(a_1,\ldots, a_m)) 
\end{equation*}
is trivial.
Therefore, we have 
 \begin{equation}
 	\mathfrak L(k,(a_1,\ldots, a_m),R) =\rho^k \mathfrak L(0,(a_1,\ldots, a_m),R) \cap (-R,R). \label{eq:fffff3333}
 \end{equation}
This equality implies that the set $\mathfrak L(k,(a_1,\ldots, a_m),R)$ is a finite set for each $k > 0$ because $\mathfrak L(0,(a_1,\ldots, a_m),R)$ is a finite set.

We want to show that there exists $N>0$ such that, for each $k>N$, the set $\mathfrak L(k,(a_1,\ldots, a_m))$ is the singleton $\{0\}$ when $0 \in \mathfrak L(0,(a_1,\ldots, a_m),R)$ and it is an empty set otherwise.
The claim is trivial when $\mathfrak L(0,(a_1,\ldots, a_m),R)$ is an empty set or coincides with $\{0\}$.
In the other case, we have $d= \inf\{|c|\;|\;  0 \neq c \in \mathfrak L(0,(a_1,\ldots, a_m),R)\}>0$ because the set $\mathfrak L(0,(a_1,\ldots, a_m),R)$ is a finite set.
Take $N>0$ so that $\rho^Nd>R$.
For any $k>N$, the set  $\mathfrak L(k,(a_1,\ldots, a_m),R)$ is either the singleton $\{0\}$ or empty by the equality (\ref{eq:fffff3333}).
Therefore, we get $$\mathfrak J_1((a_1,\ldots, a_m),R)=\bigcup_{k \geq 0} \mathfrak L(k,(a_1,\ldots, a_m),R) = \bigcup_{k=0}^N \mathfrak L(k,(a_1,\ldots, a_m),R).$$
Since $\mathfrak L(k,(a_1,\ldots, a_m),R)$ is a finite set for each $k  \geq 0$, the set $\mathfrak J_1((a_1,\ldots, a_m),R)$ is also a finite set. 
We have demonstrated that $\mathfrak I((a_1,\ldots, a_m),R)$ is a finite set. 

The final task is to find $r>0$ such that  $v(E^m) \cap (-r,r) \subseteq \{0\}$.
Take a sufficiently large $R>0$ so that $\mathfrak I((a_1,\ldots, a_m),R)=v(E^m) \cap (-R,R)$ is not empty.
Since $\mathfrak I((a_1,\ldots, a_m),R)$ is a finite set, we can take $r>0$ so that $v(E^m) \cap (-r,r)  = \mathfrak I((a_1,\ldots, a_m),R) \cap (-r,r) \subseteq \{0\}$.
\end{proof}

\begin{lemma}\label{lem:example_basic_pre2}
	Let $\mathcal L$ be a language and $P$ be a unary predicate.
	Put $\mathcal L_P=\mathcal L \cup \{P\}$.
	Let $\mathcal M=(M,\ldots)$ be an $\mathcal L$-structure and $\mathcal M_P=(\mathcal M,P)$ be an expansion of $\mathcal M$ by $P$.
	We denote the set $\{x \in M\;|\; \mathcal M_P \models P(x)\}$ by the same symbol $P$.
	Let $S \subseteq M^n$ be the set definable in $\mathcal M_P$.
	There exist $l, m \geq 0$ and an $\mathcal L$-formula $\psi(\overline{x},\overline{y}, \overline{z})$ such that the sequences of variables $\overline{y}$ and $overline{z}$ is of length $l$ and $m$, respectively, and $$S= \Pi(X \cap (M^n \times P^l \times (M \setminus P)^m)),$$
	where $\Pi:M^n \times M^l \times M^m \to M^n$ is the coordinate projection onto the first $n$ coordinates and $X=\{(\overline{x},\overline{y}, \overline{z}) \in M^n \times M^l \times M^m\;|\; \mathcal M \models \psi(\overline{x},\overline{y}, \overline{z})\}$.
\end{lemma}
\begin{proof}
	Let $\phi(\overline{x})$ be an $\mathcal L_P$-formula defining the set $S$.
	We replace the atomic $\mathcal L_P$-formulas in $\phi(\overline{x})$ of the form $P(t_i)$ with  the atomic $\mathcal L$-formulas $``y_i=t_i"$ and $\neg P(t_i)$ with  the atomic $\mathcal L$-formulas $``z_i=t_i"$, where $t_i$ are $\mathcal L$-terms and $y_i$ and $z_i$ are new variables.
\end{proof}

\begin{definition}
	Let $U$ be a nonempty open subset of $\mathbb R^n$ and $f$ be a continuous function defined on $U$.
	We say that $f$ is \textit{locally constant} at $x \in U$ if there is an open neighborhood $B$ of $x$ contained in $U$ such that the restriction of $f$ to $U$ is constant.
	It is obvious that $f$ is constant when $U$ is connected and $f$ is locally constant everywhere on $U$. 
\end{definition}

\begin{lemma}\label{lem:example_special_case}
	Let $E=\{\rho^n\;|\; n \in \mathbb Z,\ n \geq 0\}$, where $\rho$ is a real number with $\rho>1$.
	Let $\lambda_r$ be the endomorphism on $\mathbb R$ given by $x \mapsto rx$ for each $r \in \mathbb R$.
	Set $\mathcal V=(\mathbb R,0,1,+,<,(\lambda_r)_{r \in \mathbb R})$ and $\mathcal V_E=(\mathcal V, E)$.
	Any bounded set definable in $\mathcal V_E$ is definable in $\mathcal V$.
\end{lemma}
\begin{proof}
	Set $\mathcal L=(0,1,+,<,(\lambda_r)_{r \in \mathbb R})$ and $\mathcal L_E= \mathcal L \cup \{E\}$.
	Let $\mathcal V_E^*=(\mathbb R^*,0,1,+^*,<^*,(\lambda_r^*)_{r \in \mathbb R}, E^*)$ be a proper elementary extension of the $\mathcal L_E$-structure $\mathcal V_E$ with the infinitesimals.
	Set $\mu=\{x \in \mathbb R^*\;|\; |x|<r\ (\forall r \in \mathbb R_+)\}$, where $\mathbb R_+:=\{r \in \mathbb R\;|\; r>0\}$.
	Let $V$ be the $\mathbb R$-vector subspace of $\mathbb R^*$ generated by $E^*$.
	We first demonstrate the following claim:
	\medskip
	
	\textbf{Claim 1.} $V \cap \mu = \{0\}$.
	\begin{proof}[Proof of Claim 1.]
		Take $z \in V \cap \mu$.
		There exist $a_1, \ldots, a_m \in \mathbb R$ and $x_1 ,\ldots, x_m \in E^*$ such that $z=\sum_{i=1}^m \lambda_{a_i}(x_i)$.
		By Lemma \ref{lem:example_basic_pre}, there exists $0<r \in \mathbb R$ such that $$ \mathcal V_E \models \forall (x_1,\ldots, x_m) \in E^m\ \ \sum_{i=1}^m\lambda_{a_i}(x_i) = 0 \vee \left|\sum_{i=1}^m\lambda_{a_i}(x_i)\right| \geq \lambda_r(1).$$
		Since $\mathcal V_E^*$ is an elementary extension of $\mathcal V_E$, this sentence holds true in $\mathcal V_E^*$.
		It implies that either $z=0$ or $|z| \geq \lambda_r(1)$.
		We have $z=0$ because $z \in \mu$.
	\end{proof}
	
	By Zorn's lemma and Claim 1, there exists a maximal $\mathbb R$-vector subspace $W$ of $\mathbb R^*$ such that $W \cap \mu=\{0\}$ and $V \subseteq W$.
	We have $\mathbb R^*=W \oplus \mu$ as an $\mathbb R$-vector space.
	The proof of the equality $\mathbb R^*=W \oplus \mu$ is a routine, and we omit it.
		
	We next prove the following claim:
	\medskip
	
	\textbf{Claim 2.} 
	Let $S \subseteq \mathbb R^n$ be an arbitrary set definable in $\mathcal V_E$.
	There exists $\varepsilon>0$ such that $S \cap (-\varepsilon, \varepsilon)^n$ is definable in $\mathcal V$.
	\begin{proof}[Proof of Claim 2.]
	By Lemma \ref{lem:example_basic_pre2}, there exists an $\mathcal L$-formula $\psi(\overline{x},\overline{y},\overline{z})$ such that $$S= \Pi(X \cap (\mathbb R^n \times E^l \times (\mathbb R \setminus E)^m)),$$
	where $\Pi:\mathbb R^n \times \mathbb R^l  \times \mathbb R^m \to \mathbb R^n$ is the coordinate projection onto the first $n$ coordinates and $X=\{(\overline{x},\overline{y}, \overline{z}) \in \mathbb R^n \times \mathbb R^l \times \mathbb R^m\;|\; \mathcal V \models \psi(\overline{x},\overline{y},\overline{z})\}$.
	We first reduce to the case in which $X$ is the graph of a definable map $g=(g_1,\ldots, g_m) $ defined on $\Pi'(X)$, where $\Pi':\mathbb R^n \times \mathbb R^l \times \mathbb R^m \to \mathbb R^n \times \mathbb R^l$ denotes the projection onto the first $(n+l)$ coordinates.
	By the definable cell decomposition theorem \cite[Chapter 3, Theorem 2.11]{vdD}, we may assume that $X$ is a cell without loss of generality.
	By the definition of cells, there exist a coordinate projection $\pi:\mathbb R^m \to \mathbb R^{m_1}$ and a definable continuous map $f:(\Pi',\pi)(X) \to \mathbb R^{m_2}$ such that, for each $(\overline{x},\overline{y}) \in \Pi'(X)$, the fiber $X_{(\overline{x},\overline{y})}:=\{\overline{z} \in \mathbb R^m\;|\; (\overline{x},\overline{y}, \overline{z}) \in X\}$ is the graph of a definable continuous map $f(\overline{x},\overline{y},\cdot)$ defined on the definable open subset $\pi(X_{(\overline{x},\overline{y})})$ of $\mathbb R^{m_1}$ after an appropriate permutation of the coordinates.
	Here, $m_2=m - m_1$ and $(\Pi',\pi)$ denotes the projection defined by $(\Pi',\pi)(\overline{x},\overline{y},\overline{z})=(\overline{x},\overline{y},\pi(\overline{z}))$ for $\overline{x} \in \mathbb R^n$, $\overline{y} \in \mathbb R^l$ and  $\overline{z} \in \mathbb R^m$.
	We assume that $\pi$ is the coordinate projection onto the first $m_1$ coordinates for simplicity of notations.
	
	We fix $(\overline{x},\overline{y}) \in \Pi'(X)$.
	It is easy to demonstrate that $\pi(X_{(\overline{x},\overline{y})}) \cap (\mathbb R \setminus E)^{m_1}$ is a nonempty open set.
	We omit the proof.
	Let $\rho_i:\mathbb R^{m_2} \to \mathbb R$ be the projection onto the $i$-th coordinate for each $1 \leq i \leq m_2$.
	Let $Z_i=\{(\overline{x},\overline{y},\overline{z}) \in X\;|\; \rho_i \circ f(\overline{x},\overline{y} ,\cdot) \text{ is locally constant at }\pi(\overline{z})\}$.
	By considering the definable cell decomposition \cite[Chapter 3, Theorem 2.11]{vdD} partitioning $\{Z_i\;|\;1 \leq i \leq m_2\} \cup \{X\}$, we may assume that, for each $1 \leq i \leq m_2$, either $\rho_i \circ f(\overline{x},\overline{y}, \cdot):\pi(X_{(\overline{x},\overline{y})}) \to \mathbb R$ is constant for each $(\overline{x},\overline{y} ) \in \Pi'(X)$ or $\rho_i \circ f(\overline{x},\overline{y} ,\cdot)$ is not locally constant anywhere  for each $(\overline{x},\overline{y} ) \in \Pi'(X)$.
	Let $I$ be the set of indexes $1 \leq i \leq m_2$ such that $\rho_i \circ f(\overline{x},\overline{y}, \cdot)$ is constant.
	We assume that $I=\{1,2,\ldots,q\}$ for simplicity of notations.
	There exists a definable map $g_i:\Pi'(X) \to \mathbb R$ such that $g_i(\overline{x},\overline{y})=\rho_i \circ f(\overline{x},\overline{y}, \overline{u})$ for each $\overline{u} \in \pi(X_{(\overline{x},\overline{y})})$ when $i \in I$.
	We next consider the case in which $i \notin I$.
	Take a bounded open box $V$ contained in $\pi(X_{(\overline{x},\overline{y})}) \cap (\mathbb R \setminus E)^{m_1}$.
	Since the image of bounded open box under the map definable in $\mathcal V$ is bounded, the intersection of $\rho_i \circ f(\overline{x},\overline{y},V)$ with $E$ is a finite set.
	Therefore, $\{\overline{z} \in V\;|\;\rho_i \circ f(\overline{x},\overline{y},\overline{z}) \in E\}$ has an empty interior and definable in $\mathcal V$.
	It implies $\{\overline{z} \in V\;|\;\rho_i \circ f(\overline{x},\overline{y},\overline{z}) \notin E \text{ for all }i \notin I\}$ is nonempty and open by \cite[Chapter 4, Proposition 1.3(iii)]{vdD}.
	They implies that $X_{(\overline{x},\overline{y})} \cap (R \setminus E)^m \neq \emptyset$ if and only if $g_i(\overline{x},\overline{y}) \notin E$ for each $i \in I$. 
	By replacing $\{(\overline{x},\overline{y},z_1,\ldots,z_q) \in \mathbb R^n \times \mathbb R^l \times \mathbb R^q\;|\;(\overline{x},\overline{y}) \in \Pi'(X)  \text{ and } z_i = g_i(\overline{x},\overline{y})\}$ with $X$, we may assume that $X$ is the graph of a definable map $g=(g_1,\ldots, g_m) $ defined on $\Pi'(X)$.
	
	Let $\phi(\overline{x},\overline{y})$ be an $\mathcal L$-formula defining the definable set $\Pi'(X)$.
	It is well known that $\mathcal V$ admits elimination of quantifiers.
	See \cite[Chapter 1, Corollary (7.6)]{vdD}.
	Therefore, we may assume that $\phi(\overline{x},\overline{y})$ is of the form
	$$ \bigvee_{i=1}^{k_1} \bigwedge_{j=1}^{k_2} v_{ij}(\overline{x},\overline{y}) *_{ij} \lambda_{r_{ij}}(1),$$
	where $v_{ij}$ are linear functions with real coefficients and $*_{ij} \in \{=,<\}$.
	Note that a linear function with real coefficients is an $\mathcal L$-term.
	We can reduce to the case in which $$\phi(\overline{x},\overline{y})=\bigwedge_{j=1}^{k} v_{j}(\overline{x},\overline{y}) *_{j} \lambda_{r_{j}}(1),$$ 
	where $v_{j}$ are linear functions with real coefficients and $*_{j} \in \{=,<\}$, and $g_i$ are affine maps with real coefficients for $1 \leq i \leq m$ by \cite[Chapter 1, Proposition 7.4]{vdD}.
	Since $g_i$ are affine maps with real coefficients, we have $g_i(\overline{x},\overline{y}) \in W$ when $(\overline{x},\overline{y}) \in W^{n+l}$.
	
	We fix arbitrary $(\overline{a},\overline{b}) \in W^n \times W^l$, $1 \leq i \leq m$, $1 \leq j \leq k$.
	We define an $\mathcal L$-formula $\psi'_{(\overline{a},\overline{b}),j}(\overline{x},\overline{y}) $ so that
	\begin{align}
	 &((\overline{a},\overline{b})+\mu^{n+l}) \cap \{(\overline{x},\overline{y}) \in (\mathbb R^*)^n  \times (\mathbb R^*)^l\;|\; \mathcal V \models (v_{j}(\overline{x},\overline{y}) *_{j} \lambda_{r_{j}}(1))\}\label{eq:nnnn}\\
	 &= (\overline{a},\overline{b})+\{(\overline{x},\overline{y}) \in \mu^n \times \mu^l \;|\; \mathcal V \models \psi'_{(\overline{a},\overline{b}),j}(\overline{x},\overline{y})\}\nonumber.
	\end{align}
Note that we have either $v_{j}(\overline{a},\overline{b})=\lambda_{r_{j}}(1)$ or $v_{j}(\overline{a},\overline{b})-\lambda_{r_{j}}(1) \not\in \mu$ because $1 \in E^*$ and $W$ is an $\mathbb R$-vector space containing every element in $E^*$.
We first consider the case in which $v_{j}(\overline{a},\overline{b})=\lambda_{r_{j}}(1) $.
In this case, we set $$\psi'_{(\overline{a},\overline{b}),j}(\overline{x},\overline{y})=`` v_{j}(\overline{x},\overline{y})  *_{j} 0".$$
The equality (\ref{eq:nnnn}) is satisfied by linearity of $v_{j}$.
When $v_{j}(\overline{a},\overline{b}) \neq \lambda_{r_{j}}(1) $, we set 
$$\psi'_{(\overline{a},\overline{b}),j}(\overline{x},\overline{y})=
\left\{
\begin{array}{ll}
`` 0= 0" & \text{if } *_{ij}=``< " \text{ and } v_{j}(\overline{a},\overline{b})<\lambda_{r_{j}}(1),\\
``0 =1" & \text{elsewhere.}
\end{array}
\right.
$$
We obtain the equality (\ref{eq:nnnn}) because $v_{j}(\overline{a},\overline{b})-\lambda_{r_{j}}(1) \not\in \mu$  in this case.

We also set $$\psi''_{(\overline{a},\overline{b}),i}(\overline{x},\overline{y})=
\left\{
\begin{array}{ll}
	`` g_i(\overline{x},\overline{y}) \neq 0" & \text{if } g_i(\overline{a},\overline{b}) \in E^*,\\
	`` 0=0" & \text{elsewhere.}
\end{array}
\right.
$$
For any $(\overline{x},\overline{y}) \in \mu^n \times \mu^l$, the intersection of the fiber $X_{(\overline{a}+\overline{x}, \overline{b}+\overline{y})} $ with $(\mathbb R^* \setminus E^*)^m$ is not empty if and only if $\mathcal V \models \bigwedge_{i=1}^m\psi''_{(\overline{a},\overline{b})}(\overline{x},\overline{y})$ because $X$ is the graph of a definable map $g=(g_1,\ldots, g_m)$ defined on $\Pi'(X)$.

We get
\begin{align}
	&((\overline{a},\overline{b})+\mu^{n+l}) \cap \Pi'( \{(\overline{x},\overline{y}) \in (\mathbb R^* )^n \times (\mathbb R^*)^l \;|\;\exists \overline{z} \in (\mathbb R^* \setminus E^*)^m,\  \mathcal V \models \psi(\overline{x},\overline{y},\overline{z})\})\label{eq:nnnn3}\\
	&= (\overline{a},\overline{b})+\{(\overline{x},\overline{y}) \in \mu^n \times \mu^l\;|\; \mathcal V \models \bigwedge_{j=1}^{k}\psi'_{(\overline{a},\overline{b}),j}(\overline{x},\overline{y}) \wedge \bigwedge_{i=1}^m \psi''_{(\overline{a},\overline{b}),i}(\overline{x},\overline{y}) \}\nonumber.
\end{align}

Note that $\psi'_{(\overline{a},\overline{b}),j}(\overline{x},\overline{y})$ is one of three formulas $`` v_{j}(\overline{x},\overline{y})  *_{j} 0"$, $`` 0= 0"$  and $``0 =1"$ and these three formulas are free from the coordinates of $(\overline{a},\overline{b}) \in W^n \times W^l$.
Similar claims hold for $\psi''_{(\overline{a},\overline{b}),i}(\overline{x},\overline{y},z_i)$.
They imply that there are finitely many $\mathcal L$-formulas $\widetilde{\psi_1}(\overline{x},\overline{y}), \ldots, \widetilde{\psi_p}(\overline{x},\overline{y})$ and a map $\iota:W^n  \times W^l\to \{1 ,\ldots, p\}$ such that 
	\begin{align}
	&((\overline{a},\overline{b})+\mu^{n+l}) \cap \{(\overline{x},\overline{y}) \in (\mathbb R^*)^n \times (\mathbb R^*)^l \;|\; \exists \overline{z} \in (\mathbb R^* \setminus E^*)^m,\ \mathcal V \models \psi(\overline{x},\overline{y},\overline{z})\}\label{eq:nnnn2}\\
	&= (\overline{a},\overline{b})+\{(\overline{x},\overline{y}) \in \mu^n \times \mu^l\;|\; \mathcal V \models \widetilde{\psi_{\iota(\overline{a},\overline{b})}}(\overline{x},\overline{y})\}\nonumber.
\end{align}

Set $J=\iota(\overline{0}_n,(E^*)^l)$, where $\overline{0}_n$ is the sequence of $n$-copies of zeros.
Let $X^*$ and $S^*$ be the sets definable in $\mathcal V_E$ defined by the formulas defining $X$ and $S$, respectively.
We have 
\begin{align*}
S^* \cap \mu^n &= \Pi(X^* \cap (\mu^n \times (E^*)^l \times  (\mathbb R^* \setminus \mathbb E^*)^m))\\
&=\bigcup_{i \in J}\{\overline{x} \in \mu^n\;|\; \mathcal V \models \psi'_i(\overline{x},\overline{0}_l)\}\\
&=\left\{\overline{x} \in \mu^n\;\left|\; \mathcal V \models \bigvee_{i \in J} \psi'_i(\overline{x},\overline{0}_l)\right.\right\}.
\end{align*}

Set $\Psi(\overline{x})=\bigvee_{i \in J} \psi'_i(\overline{x},\overline{0}_l)$.
Since $I$ is a finite set, $\Psi(\overline{x})$ is an $\mathcal L$-formula.
Take $\varepsilon \in \mu$ with $\varepsilon>0$.
We have $S \cap (-\varepsilon,\varepsilon)^n = \{\overline{x} \in (-\varepsilon,\varepsilon)^n\;|\; \mathcal V_E^* \models \Psi(\overline{x}) \}$.
	Let $\phi(\overline{x})$ be an $\mathcal L_E$-formula defining the set $S$.
	We have demonstrated that
	$$\mathcal V_E \models \exists \varepsilon>0\ \forall \overline{x}=(x_1,\ldots, x_n)\ \left(\phi(\overline{x}) \wedge \bigwedge_{i=1}^n |x_i|<\varepsilon\right) \leftrightarrow \left(\Psi(\overline{x}) \wedge \bigwedge_{i=1}^n |x_i|<\varepsilon\right).$$
	This sentence holds true in $\mathcal V_E$ because $\mathcal V_E$ is elementarily equivalent to $\mathcal V_E^*$. 
	It implies that $S \cap (-\varepsilon,\varepsilon)^n$ is definable in $\mathcal V$ for some $\varepsilon >0$.
	\end{proof}
	
	Let $S$ be a bounded subset of $\mathbb R^n$ definable in $\mathcal V_E$.
	We want to show that $S$ is definable in $\mathcal V$.
	Take $R>0$ so that $S \subseteq [-R,R]^n$.
	We set $B(\overline{x},\varepsilon):=\{ (y_1,\ldots, y_n) \in \mathbb R^n\;|\; |y_i-x_i|<\varepsilon \ (1 \leq i \leq n)\}$ for any $\overline{x}=(x_1,\ldots,x_n) \in \mathbb R^n$ and $\varepsilon>0$. 
	Fix an arbitrary $\overline{x}\in [-R,R]^n$.
	 Apply Claim 2 to the set $\{\overline{y} \in \mathbb R^n\;|\; \overline{x}+\overline{y} \in S\}$ which is definable in $\mathcal V_E$.
	There exists $\varepsilon_{\overline{x}}>0$ such that $S \cap B(\overline{x},\varepsilon_{\overline{x}})$ is definable in $\mathcal V$.
	Since $[-R,R]^n$ is compact, there are finitely many points $\overline{x}_1,\ldots, \overline{x}_m$ such that $[-R,R^n] \subseteq \bigcup_{i=1}^m  B(\overline{x}_i,\varepsilon_{\overline{x}_i})$.
	We have 
	$S=\bigcup_{i=1}^m S \cap B(\overline{x}_i,\varepsilon_{\overline{x}_i})$.
	It implies that $S$ is definable in $\mathcal V$ because $S \cap B(\overline{x}_i,\varepsilon_{\overline{x}_i})$ is definable in $\mathcal V$ for each $1 \leq i \leq m$.
\end{proof}

\begin{theorem}\label{thm:example}
	Let $\mathcal R$ be a semi-bounded o-minimal expansion of the o-minimal structure $(\mathbb R, <, +, 0, 1, (\lambda_r)_{r \in \mathbb R})$.
	Here, $\lambda_r$ is the endomorphism given by $\mathbb R \ni x \mapsto rx \in \mathbb R$ for each $r \in \mathbb R$.
	Let $E=\{\rho^n\;|\; n \in \mathbb Z,\ n \geq 0\}$, where $\rho$ is a real number with $\rho>1$.
	Let $\mathcal R_E$ be the expansion $(\mathcal R,E)$ of $\mathcal R$ by $E$.
	A bounded set definable in $\mathcal R_E$ is definable in $\mathcal R$.
	In particular,  the structure $\mathcal R_E$ is a locally o-minimal structure.
\end{theorem}
\begin{proof}
	Let $S \subseteq \mathbb R^n$ be a bounded set definable in $\mathcal R_E$.
	We want to show that $S$ is definable in $\mathcal R$.
	Let $\mathcal L$ be the language of the structure $\mathcal R$ and set $\mathcal L_E=\mathcal L \cup \{E\}$.
	
	By Lemma \ref{lem:example_basic_pre2}, we get an $\mathcal L$-formula $\psi(\overline{x},\overline{y},\overline{z})$ such that
	\begin{equation}
		S=\Pi(X \cap (\mathbb R^n \times E^l \times (\mathbb R \setminus E)^m)),\label{eq:fff}
	\end{equation}
	where $X=\{(\overline{x},\overline{y},\overline{z}) \in \mathbb R^n \times \mathbb R^l \times \mathbb R^m \;|\; \mathcal R \models \psi(\overline{x},\overline{y},\overline{z})\}$ and $\Pi: \mathbb R^n \times \mathbb R^l \times \mathbb R^m \to \mathbb R^n$ be the coordinate projection onto the first $n$ coordinates.
	
	Note that $X$ is definable in the semi-bounded o-minimal structure $\mathcal R$.
	By \cite[Fact 1.6]{Edmundo}, the set $X$ is partitioned into finitely many normalized cones $X_1, \ldots, X_k$ definable in $\mathcal R$.
	Here, a \textit{cone} is either a bounded set definable in $\mathcal R$ or a subset of $\mathbb R^{n+l+m}$ of the form $$\operatorname{Cone}(B,\{\overline{v}_1,\ldots,\overline{v}_p\}):=\left\{\left.\overline{b}+\sum_{i=1}^p t_i\overline{v}_i \;\right|\; \overline{b} \in B,\ t_i >0 \ (1 \leq i \leq p)\right\},$$
	where $B$ is a bounded subset of $\mathbb R^{n+l+m}$ definable in $\mathcal R$ and $\{\overline{v}_1,\ldots,\overline{v}_p\}$ is a nonempty family of linearly independent vectors in $\mathbb R^{n+m}$.
	We set $\operatorname{Cone}(B, \emptyset)=B$.
	The cone $\operatorname{Cone}(B,\{\overline{v}_1,\ldots,\overline{v}_p\})$ is called \textit{normalized} if, for any $\overline{x} \in \operatorname{Cone}(B,\{\overline{v}_1,\ldots,\overline{v}_p\})$, there exist unique $\overline{b} \in B$ and $t_i>0$ such that $\overline{x}=\overline{b}+\sum_{i=1}^p t_i\overline{v}_i$.
	(We do not use normalizedness in the proof.)
	We obviously have $S=\bigcup_{i=1}^k \Pi(X_i \cap (\mathbb R^n \times E^l \times (\mathbb R \setminus E)^m))$.
	If $\Pi(X_i \cap  (\mathbb R^n \times E^l \times (\mathbb R \setminus E)^m))$ is definable in $\mathcal R$ for each $1 \leq  i \leq k$, the set $S$ is also definable in $\mathcal R$.
	Therefore, we may assume without loss of generality that $X$ is a normalized cone $\operatorname{Cone}(B,\{\overline{v}_1,\ldots,\overline{v}_p\})$ for some $B$ and some possibly empty family $\{\overline{v}_1,\ldots,\overline{v}_p\}$, and the equality (\ref{eq:fff}) still holds true.
	We fix such $B$ and $\{\overline{v}_1,\ldots,\overline{v}_p\}$.
	
	Let $\mathcal V$ and $\mathcal V_E$ be the structures given in Lemma \ref{lem:example_special_case}.
	Consider the set
	$$Y=\{(\overline{x},\overline{u}) \in \mathbb R^n \times \mathbb R^{n+l+m}\;|\; \exists \overline{y} \in E^l \ \exists \overline{z} \in (\mathbb R \setminus E)^m\ \ (\overline{x},\overline{y},\overline{z}) \in \operatorname{Cone}(\{\overline{u}\},\{\overline{v}_1,\ldots,\overline{v}_p\})\}.$$
	Note that $Y$ is definable in $\mathcal V_E$.
	Let $\Pi':\mathbb R^n \times \mathbb R^{n+l+m} \to \mathbb R^n$ be the coordinate projection onto the first $n$ coordinates.
	Take $R>0$ so that $S \subseteq (-R,R)^n$ and $B \subseteq  (-R,R)^{n+l+m}$.
	Such an $R$ exists because both $S$ and $B$ are bounded.
	We have 
	\begin{align*}
		S &= S \cap  (-R,R)^n = \Pi(X \cap ((-R,R)^n \times E^l \times (\mathbb R \setminus E)^m) )\\
		&=\{\overline{x} \in (-R,R)^n\;|\; \exists \overline{y} \in E^l \ \exists \overline{z} \in (\mathbb R \setminus E)^m\ (\overline{x},\overline{y},\overline{z}) \in \operatorname{Cone}(B,\{\overline{v}_1,\ldots,\overline{v}_p\}) \}\\
		&=\Pi'(\{(\overline{x},\overline{u}) \in (-R, R)^n\times B\;|\; \exists \overline{y} \in E^l \ \exists \overline{z} \in (\mathbb R \setminus E)^m\ \ (\overline{x},\overline{y},\overline{z}) \in \operatorname{Cone}(\{\overline{u}\},\{\overline{v}_1,\ldots,\overline{v}_p\})\})\\
		&=\Pi'(Y \cap ((-R,R)^n \times B))\\
		&=\Pi'(Y \cap ((-R,R)^n \times (-R,R)^{n+l+m}) \cap (\mathbb R^n \times B)).
	\end{align*}
Since $Y \cap ((-R,R)^n \times (-R,R)^{n+l+m})$ is definable in $\mathcal V_E$ and bounded, it is definable in $\mathcal V$ by Lemma \ref{lem:example_special_case}. 
The set $Z:=Y \cap ((-R,R)^n \times (-R,R)^{n+l+m}) \cap (\mathbb R^n \times B)$ is definable in $\mathcal R$ because $\mathcal R$ is an expansion of $\mathcal V$ and $B$ is definable in $\mathcal R$.
The set $S$ is also definable in $\mathcal R$ because $S$ is the projection image of $Z$.

The `in particular' part of the theorem immediately follows from what we have just proven.
\end{proof}

\begin{remark}
We can say more than Theorem \ref{thm:example} if we look its proof and the proof of Lemma \ref{lem:example_special_case} closely. 
Friedman and Miller gave a sufficient condition for an expansion of o-minimal structure by a subset of $\mathbb R$ to be locally o-minimal in \cite[Theorem A]{FM2}.
The following claim is its refinement.
	
Let $\mathcal R$ be an o-minimal expansion of the  structure $(\mathbb R, <, +, 0, 1, (\lambda_r)_{r \in \mathbb R})$.
Let $E$ be a subset of $\mathbb R$ of cardinality $\aleph_0$.
The expansion $\mathcal R_E=(\mathcal R,E)$ is locally o-minimal if and only if the following conditions are satisfied:
\begin{enumerate}
	\item[(1)] The o-minimal structure $\mathcal R$ is semi-bounded.
	\item[(2)] For any linear function with real coefficients $v$ and for any $R>0$, the set $v(E^m) \cap (-R,R)$ is a finite set.   
\end{enumerate}
Note that condition (2) implies that $E$ is closed and discrete.
\end{remark}
\begin{proof}
	We only use the fact that the set $v(E^m) \cap (-R,R)$ is a finite set in the proofs of Lemma \ref{lem:example_special_case} and Theorem \ref{thm:example}.
	Therefore, the `if' part is proven in the same manner as Lemma \ref{lem:example_special_case} and Theorem \ref{thm:example}.
	
	We next prove the contraposition of `only if' part.
	If condition (2) is not satisfied, the structure $\mathcal R_E$ is not locally o-minimal because the set $v(E^m) \cap (-R,R)$ is definable in $\mathcal R_E$ and it is a countably infinite set for some $R>0$.
	
	Consider the case in which condition (1) is not satisfied, but condition (2) is satisfied.
	The set $E$ is unbounded.
	We may assume that $E \cap (c,\infty)$ is an infinite set for any $c \in \mathbb R$ without loss of generality.
	By \cite[Fact 1.6]{Edmundo}, there exists a bijection $f$ definable in $\mathcal R$ between an unbounded set $X$ definable in $\mathcal R$ and a bounded set $Y$ definable in $\mathcal R$.
	Let $\mathbb R^m$ and $\mathbb R^n$ be the ambient spaces of $X$ and $Y$, respectively.
	There exists a coordinate projection $\pi:\mathbb R^m \to \mathbb R$ such that $\pi(X)$ is unbounded because $X$ is unbounded.
	We may assume that $\pi(X)$ contains an open interval of the form $(c,\infty)$ because $\mathcal R$ is o-minimal expansion of an ordered group.
	By the definable choice lemma \cite[Chapter 6, Proposition (1.2)]{vdD} , there exists a map $\tau:\pi(X) \to X$ definable in $\mathcal R$ such that the composition $\pi \circ \tau$ is the identity map on $\pi(X)$.
	Let $p_i:\mathbb R^n \to \mathbb R$ be the coordinate projection onto the $i$-th coordinate for each $1 \leq i \leq n$.
	By the monotonicity theorem \cite[Chapter 3, Theorem (1.3)]{vdD}, we may assume that the restriction $p_i \circ f \circ \tau|_{(c,\infty)}$ of $p_i \circ f \circ \tau$ to $(c,\infty)$ is monotone and continuous for each $1 \leq i \leq n$.
	Since $f \circ \tau$ is injective, $g:=p_i \circ f \circ \tau|_{(c,\infty)}$ is strictly monotone for some $1 \leq i \leq n$.
	The map $g$ is a bijection definable in $\mathcal R$ between a unbounded open interval and a bounded open interval. 
	 
	The set $g(E \cap (c,\infty))$ is a bounded countably infinite set definable in $\mathcal R_E$.
	It implies that $\mathcal R_E$ is not locally o-minimal.
\end{proof}

A function $f: \mathbb R^n \to \mathbb R$  is called a \textit{restricted analytic function} if the restriction of $f$ to $[0,1]^n$ coincides with the restriction  of a real analytic function defined on a neighborhood of $[0,1]^n$ and it is zero on $\mathbb R^n \setminus [0,1]^n$.

\begin{corollary}
	The structure $(\mathbb R, <,0,1,+,(\lambda_r)_{r \in \mathbb R}, (f)_{f \in \mathcal F},E)$ is a locally o-minimal structure.
	Here, $\lambda_r$ is the endomorphism given by $\mathbb R \ni x \mapsto rx \in \mathbb R$ for each $r \in \mathbb R$.
	The family $\mathcal F$ is the collection of all restricted analytic functions.
	We set $E=\{\rho^n\;|\; n \in \mathbb Z,\ n \geq 0\}$, where $\rho$ is a real number with $\rho>1$.
\end{corollary}
\begin{proof}
	The structure $\mathcal R:=(\mathbb R, <,0,1,+,(\lambda_r)_{r \in \mathbb R}, (f)_{f \in \mathcal F})$ is obviously a reduct of the restricted analytic field $\mathbb R_{\text{an}}:=(\mathbb R, <,0,1,+,\cdot, (f)_{f \in \mathcal F})$.
	Here, the binary operation $\cdot$ is the multiplication in $\mathbb R$.
	It is well known that $\mathbb R_{\text{an}}$ is o-minimal.
	See \cite{vdDM} for instance.
	Therefore, the reduct $\mathcal R$ is also o-minimal.
	The structure $\mathcal R$ is obviously semi-bounded.
	The corollary follows from Theorem \ref{thm:example}.
\end{proof}

\section{Quasi-periodic locally o-minimal structures}\label{sec:periodic}
We first define quasi-periodic locally o-minimal structures.
\begin{definition}
Consider a locally o-minimal expansion $\mathcal R$ of the ordered group of reals.
It is called \textit{periodic} if the set $r\mathbb Z$ is definable in it for some $r>0$.
It is called \textit{quasi-periodic} if it is a reduct of periodic o-minimal expansion of the ordered group of reals.
\end{definition}
A periodic locally o-minimal expansion $\mathcal R$ of the ordered group of reals is a simple product of $r\mathbb Z$ and $[0,r)_{\text{def}}$ by \cite[Theorem 25]{KTTT}.
See \cite{KTTT} for the definition of simple products.
Any set definable in the periodic locally o-minimal structure $\mathcal R$ is of the form 
\begin{equation}
\bigcup_{a \in Z}a+X,\label{eq:ddddd}
\end{equation}
where $Z$ is a definable subset of $(r\mathbb Z)^n$ and $X$ is a subset of $[0,r)^n$ definable in some o-minimal structure by \cite[Lemma 17]{KTTT}.
Here, $a+X$ denotes the set $\{a+x\;|\; x \in X\}$.
Therefore, sets definable in a quasi-periodic o-minimal expansion of the ordered group of reals are also of the form (\ref{eq:ddddd}).
Walsberg gave a necessary and sufficient condition that an o-minimal structure is quasi-periodic in \cite{Walsberg2} which discusses a generalization of \cite{BC}.

\begin{proposition}\label{prop:quasi-periodic1}
Let $r$ be an irrational number.
Any locally o-minimal expansion of $(\mathbb R,<,0,+,\lambda_r)$ is not quasi-periodic, where $\lambda_r$ is the automorphism given by $x \mapsto rx$.
In particular, the locally o-minimal structures constructed in Theorem \ref{thm:example} are not quasi-periodic.
\end{proposition}
\begin{proof}
	Consider a locally o-minimal expansion $\mathcal R$ of $(\mathbb R,<,0,+,\lambda_r)$.
	Assume for contradiction that $\mathcal R$ is quasi-periodic.
	The expansion $\mathcal R_{s\mathbb Z}:=(\mathcal R,s\mathbb Z)$ is locally o-minimal for some $s>0$.
	Consider the set $X=\{x \in [0,r)\;|\; \exists z_1, z_2 \in \mathbb Z,\ x+sz_1=rsz_2 \}$.
	It is definable in the locally o-minimal structure $\mathcal R_{s\mathbb Z}$.
	
	We next consider the map $f:s\mathbb Z \to X$ given by $f(t)=rt-\sup\{u \in s\mathbb Z\;|\; u \leq rt\}$.
	It is also definable in $\mathcal R_{s\mathbb Z}$.
	We show that $f$ is bijective.
	It is obvious from the definitions that $f$ is surjective.
	Fix $t_1,t_2 \in s\mathbb Z$ such that $f(t_1)=f(t_2)$.
	Take $z_i \in \mathbb Z$ with $t_i=sz_i$ for $i=1,2$.
	We have $rt_1-rt_2 =sz$ for some $z \in \mathbb Z$.
	We get $r=\frac{z}{z_1-z_2}$ if $z_1 \neq z_2$.
	It implies that $r$ is a rational number.
	It is a contradiction.
	We have $z_1=z_2$ and $t_1=t_2$.
	We have proven that $f$ is bijective.
	
	It implies that the bounded set $X$ definable in $\mathcal R_{\mathbb Z}$ is of cardinality $\aleph_0$.
	It is a contradiction because  $\mathcal R_{\mathbb Z}$  is locally o-minimal.
	
	The `in particular' part is obvious because the structures constructed in Theorem \ref{thm:example} are locally o-minimal expansions of $(\mathbb R,<,0,+,\lambda_r)$ and any reduct of quasi-periodic locally o-minimal structure is again quasi-periodic.
\end{proof}

\begin{remark}
	Stronger results on the structure $(\mathbb R,<,0,+,\lambda_r, \mathbb Z)$  than Proposition \ref{prop:quasi-periodic1} were obtained by Hieronymi, Tychonievich and Walsberg in \cite{H, HT, Walsberg}.
	This structure is interdefinable with $\overline{\mathbb R}_{\mathbb Z}:=(\mathbb R,<,0,1,+,\cdot,\mathbb Z)$ by \cite[Theorem B]{HT} when $r$ is a non-quadratic irrational number.
	It is well known that a subset of $\mathbb R^n$ is definable in $\overline{\mathbb R}_{\mathbb Z}$ if and only if it is projective. 
	See \cite{K} for the definition of projective sets.
\end{remark}

\begin{proposition}\label{prop:quasi-periodic2}
	Let $E=\{\rho^n\;|\; n \in \mathbb Z,\ n \geq 0\}$, where $\rho$ is a real number with $\rho>1$.
	The  locally o-minimal structure $(\mathbb R, <, +, 0, E)$ is quasi-periodic if and only if $\rho$ is an integer.
\end{proposition}
\begin{proof}
	Let $\mathcal R=(\mathbb R, <, +, 0, E)$.
	It is locally o-minimal structure because it is a reduct of the structure given in Theorem \ref{thm:example}.
	We have only to demonstrate that $\mathcal R_{r\mathbb Z}=(\mathcal R,r\mathbb Z)$ is locally o-minimal for some $r>0$ if and only if $\rho$ is an integer.
	
	When $\rho$ is an integer, $\mathcal R_{\mathbb Z}=(\mathcal R,\mathbb Z)$ is locally o-minimal.
	The proof is similar to that of \cite[Example 20]{KTTT}.
	We omit the details.
	
	We next consider the remaining case.
	Assume for contradiction that $\mathcal R_{r\mathbb Z}$ is locally o-minimal for some $r>0$.
	Consider the map $f:E \to [0,r)$ defined by $f(x)=x - \sup\{t \in r\mathbb Z\;|\; t \leq x\}$.
	It is definable in $\mathcal R_{r\mathbb Z}$.
	Therefore, the image of $f$ is a finite set because $\mathcal R_{r\mathbb Z}$ is locally o-minimal.
	There exists $y^* \in f(E)$ whose inverse image under $f$ is an infinite set.
	Fix such $y^*$ and set $F=f^{-1}(y^*)$.
	
	We consider two separate cases.
	The first case is the case in which $\rho$ belongs to $\mathbb Q\setminus \mathbb Z$.
	There are two coprime integers $p$ and $q$ such that $\rho=p/q$ and $q>1$.
	If $\rho^m,\rho^n \in F$ and $m<n$, we have $\rho^n-\rho^m=rk$ for some $k \in \mathbb Z$.
	We have 
	\begin{equation}
	\dfrac{p^n-p^mq^{n-m}}{q^n}=rk. \label{eq:rational_eq}
	\end{equation}
	It implies that $r$ is a rational number.
	There are two coprime integers $p'$ and $q'$ such that $r=p'/q'$.
	Since $p^n-p^mq^{n-m}$ and $q^n$ are coprime, $q'$ is divisible by $q^n$.
	Since $F$ is an infinite set, $q'$ should be divisible by $q^n$ for all $n>0$.
	It is a contradiction.
	
	We next consider the case in which $\rho$ is an irrational number.
	Set $G=\{(t_1,t_2) \in F^2\;|\; t_1<t_2\}$ and consider the map $g:G \to r\mathbb Z$ given by $g(t_1,t_2)=t_2-t_1$.
	Note that $G$ is infinite because $F$ is infinite.
	We put $H=g(G)$, which is definable in $\mathcal R_{r\mathbb Z}$.
	When $H$ is an infinite set, the set $\rho H$ is an infinite subset of $\rho (r\mathbb Z)$ definable in $\mathcal R_{r\mathbb Z}$.
	We lead to a contradiction in the same manner as the proof of Proposition \ref{prop:quasi-periodic1} using the assumption that $\rho$ is irrational.
	We omit the details.
	
	We assume that $H$ is a finite set.
	There exists $z^* \in H$ such that $g^{-1}(z^*)$ is infinite because $G$ is infinite.
	Fix such an element $z^*$ and set $I=g^{-1}(z^*)$.
	Let $\pi_i:I \to E$ be the map defined by $\pi_i(t_1,t_2)=t_i$ for $i=1,2$.
	Note that $y_1-x_1=y_2-x_2$ for any $(x_1,y_1), (x_ 2,y_2) \in I$.
	It is obvious that, if $(x_1,y_1), (x_ 2,y_2) \in I$ and $x_1=x_2$ holds true, we have $y_1=y_2$, and $(x_1,y_1)=(x_2,y_2)$.
	It implies that $\pi_1$ is injective.
	It is also obvious that $\pi_2$ is injective. 
	
	Fix an arbitrary $(x,y) \in I$.
	Since $I$ is an infinite set and $\pi_1$ is injective, there is $(x',y') \in I$ with $x'>y$.
	We have $y'-x'=y-x$ and $y'/y-x'/y = 1 -x/y$.
	We get $x/y = 1 +x'/y-y'/y$.
	Remark that $x'/y$ and $y'/y$ are elements in $E$ because $E$ is a geometric progression and $y<x'<y'$.
	Consider the map $h:E^2 \to \mathbb R$ given by $h(t_1,t_2)=1+t_1-t_2$.
	The set $h(E^2)$ is definable in $\mathcal R$ and countable.
	The equality $x/y = 1 +x'/y-y'/y$ implies that $J:=\{x/y \in \mathbb R\;|\; (x,y) \in I\}$ is a subset of $h(E^2)$.
	We have $x_1/y_1 \neq x_2/y_2$ if $(x_1,y_1), (x_ 2,y_2) \in I$  and $(x_1,y_1) \neq (x_ 2,y_2)$ because $\pi_2$ is injective.
	Therefore, $J$ is a countably infinite set because $I$ is so.
	The intersection $h(E^2) \cap [0,1)$ is countably infinite because it contains the set $J$.
	It contradicts the fact that $\mathcal R$ is locally o-minimal.
\end{proof}

\appendix
\section{Expansions by fast sequences}\label{sec:appendix}
Theorem 3 of \cite{FM} says that $(\mathcal R,\phi)^{\#}$ is locally o-minimal if $\mathcal R$ is semi-bounded o-minimal structure and $\phi=(\phi_k)$ is strictly increasing sequence of positive real numbers such that $$\lim_{k \to \infty}\phi_k/\phi_{k+1}=0.$$
Here, $(\mathcal R,\phi)^{\#}$ denotes the structure $(\mathcal R,(S))$, where $S$ ranges over all subsets of Cartesian powers $\phi^n$.
The authors of \cite{FM} used the fact that the image of $\phi^n$ under affine linear functions $\mathbb R^n \to \mathbb R$ are closed and discrete without proof.
This is a gap of the original proof.
We prove it in this appendix.
The following proposition implies that the image is discrete and closed.

\begin{proposition}
	Let $\phi=(\phi_k)$ be a strictly increasing sequence of positive real numbers such that $ \lim_{k \to \infty}\phi_k/\phi_{k+1}=0$.
	For a positive integer $m$, a sequence of real numbers $(a_1, \ldots, a_m)$ of length $m$ and a positive real number $R$, we put
	$$\mathfrak I((a_1,\ldots, a_m),R)= v(\phi^m) \cap (-R,R),$$
	where $v$ is the linear map given by $v(x_1,\ldots, x_m) = \sum_{i=1}^m a_ix_i$.
	Then the set $\mathfrak I((a_1,\ldots, a_m),R)$ is a finite set 
\end{proposition}
\begin{proof}
	Set $\mathbf{a}=(a_1,\ldots, a_m)$ for simplicity.
	For any nonempty subset $J \subseteq \{1,\ldots, m\}$, we set 
	\begin{align*}
	\mathfrak  I_J(\mathbf{a},R)&=\left\{\sum_{i=1}^l a_i \phi_{k_i} \in \mathfrak I(\mathbf{a},R)\;\middle|\; \forall i, j \in J\  k_i=k_j \text{ and }\forall i \notin J, \forall j \in J\ k_i<k_j \right\}.
	\end{align*}
	We obviously have $\mathfrak I(\mathbf{a},R)=\bigcup_{\emptyset \neq J \subseteq \{1,\ldots, l\}}\mathfrak I_J(\mathbf{a},R)$.
	Therefore, we have only to demonstrate that $\mathfrak I_J(\mathbf{a},R)$ is a finite set for a fixed subset $J  \subseteq \{1,\ldots, m\}$.
	
	We prove it by induction on $m$.
	The case in which $m=1$ is obvious.
	We next consider the case in which $m>1$.
	We may assume that there exists $k$ such that $j>k$ if and only if $j \in J$ by permuting the coordinates if necessary.
	We consider two separate cases.
	
	When $\sum_{j>k}a_j=0$, we have $\mathfrak I_J(\mathbf{a},R)=\mathfrak I((a_1,\ldots, a_k),R)$.
	The set  $\mathfrak I_J(\mathbf{a},R)$ is a finite set by the induction hypothesis.
	
	The remaining case is the case in which $\sum_{j>k}a_j \neq 0$.
	Set $b=\left|\sum_{j>k}a_j\right|$.
	We take a positive real number $\varepsilon$ so that $\varepsilon <b$.
	Because $ \lim_{k \to \infty}\phi_k/\phi_{k+1}=0$, there exists $N_1>0$ such that $\sum_{j=1}^k |a_j|\phi_{n-1}<\varepsilon\phi_n$ for each $n>N_1$.
	Take $N_2>0$ such that $(b-\varepsilon)\phi_n>R$ for each $n>N_2$.
	Such an $N_2$ exists because $\lim_{n \to \infty}\phi_n=\infty$.
	Set $N=\max\{N_1,N_2\}$.
	
	We show that $\left|\sum_{i=1}^l a_i \phi_{k_i}\right| >R$ when $n>N$,  $k_i=n$ for $i>k$ and $k_i<n$ otherwise.
	In fact, we have 
	\begin{align*}
	\left|\sum_{i=1}^l a_i \phi_{k_i}\right| &\geq \left|\sum_{j>k}a_j\right|\phi_n-\sum_{j=1}^k |a_j|\phi_{k_j} \geq  b \phi_n-\sum_{j=1}^k |a_j|  \phi_{n-1}>(b-\epsilon)\phi_n>R.
	\end{align*}
	This inequality implies that,  if $\sum_{i=1}^l a_i \phi_{k_i} \in \mathfrak I_J(\mathbf{a},R)$, then $k_i \leq N$ for each $1 \leq i \leq m$.
	It means that $\mathfrak I_J(\mathbf{a},R)$ is a finite set.
\end{proof}

\section*{Acknowledgment}
The author appreciates Kota Takeuchi and Akito Tsuboi for suggesting the problem studied in this paper.
He also appreciates Chris Miller for the discussion on Appendix \ref{sec:appendix}.

\end{document}